\documentclass[12pt]{article}
\usepackage{type1cm}
\usepackage[T1]{fontenc}

\usepackage{amssymb,amsthm,amsmath}

\usepackage{latexsym}
\usepackage{makeidx}
\usepackage{showidx}
\usepackage[matrix,arrow,cmtip]{xypic}
\usepackage{colortbl}%paquete de color
\usepackage{fancybox}%paquete para las cajas
\usepackage{fancyhdr}
\usepackage{wrapfig}%figuras rodeadas de texto
\usepackage{graphicx}

\newfont{\bb}{msbm10 at 12pt}
\newfont{\bbp}{msbm10 at 9pt}

\newtheorem{teorema}{Theorem}

\newtheorem{lema}{Lemma}
\newtheorem{corolario}{Corollary}
\newtheorem{definicion}{Definition}

\newtheorem{nota}{Remark}
\newtheorem{ejemplo}{Example}

\def\n{\hbox{\bb N}}

\def\r{\hbox{\bb R}}

\def\fl{\longrightarrow}

\def\c{\hbox{\bb C}}

\newcommand{\gt}{\widetilde{g}}
\newcommand{\Gt}{\widetilde{G}}
\newcommand{\ep}{\varepsilon}

\def\cC{\mathcal{C}}

\def\cU{\mathcal{U}}

\def\cQ{\mathcal{Q}}

\let\8=\infty \let\0=\emptyset  
\let\hat=\widehat

\let\tilde=\widetilde

\let\landa=\lambda
\let\alfa=\alpha

\let\parc=\partial
\let\ep=\varepsilon

\def\landa{\lambda}

\def\lap{\Delta}
\def\flecha{\rightarrow}

\def\cte.{\mathop{\rm cte.}\nolimits}

 \def\Im{\mathop{\rm Im }\nolimits}

\def\ln{\mathop{\rm ln }\nolimits}

\def\N{\mathbb{N}}

\def\R{\mathbb{R}}

\def\C{\mathbb{C}}
\def\D{\mathbb{D}}

\def\S{\mathbb{S}}

\newcommand{\beq}{\begin{equation}}
\newcommand{\eeq}{\end{equation}}
\begin{document}

\begin{center}
\rule{15cm}{1.5pt}\vspace{1cm}

{\LARGE \bf The geometric Neumann problem\\[5mm] for the Liouville equation}\\
\vspace{.5cm}

{\bf José A. Gálvez$^a$, Asun Jiménez${}^{a}$ and Pablo
Mira${}^{b}$}\vspace{.5cm}

\rule{15cm}{1.5pt}
\end{center}

\noindent{\small ${}^a$Departamento de Geometr\'\i a y Topolog\'\i
a, Facultad de Ciencias, Universidad de Granada, E-18071 Granada,
Spain;
e-mails: jagalvez@ugr.es, asunjg@ugr.es \\
${}^b$Departamento de Matemática Aplicada y Estadística, Universidad
Politécnica de Cartagena, E-30203 Cartagena (Murcia), Spain; e-mail:
pablo.mira@upct.es}

%\begin{flushright}
%9 de febrero de 2009.
%\end{flushright}

\begin{abstract}
Let $\Omega$ denote the upper half-plane $\R_+^2$ or the upper
half-disk $D_{\ep}^+\subset \R_+^2$ of center $0$ and radius $\ep$.
In this paper we classify the solutions $v\in
C^2(\overline{\Omega}\setminus\{0\})$ to the Neumann problem
$$ \left\{
\begin{array}{lll}
\Delta v+2 K e^v=0&\qquad\qquad&\text{in }\Omega\subseteq \r^2_+=\{(s,t)\in\r^2:\ t>0\},\\[3mm]
{\displaystyle \frac{\partial v}{\partial t}}=c_1e^{v/2}&&\text{on }\partial\Omega\cap\{s>0\},\\[3mm]
{\displaystyle \frac{\partial v}{\partial t}}=c_2e^{v/2}&&\text{on
}\partial\Omega\cap\{s<0\},
\end{array}
\right.
$$ where $K,c_1,c_2\in \R$, with the finite energy condition
$\int_{\Omega} e^v <\8$. As a result, we classify the conformal
Riemannian metrics of constant curvature and finite area on a
half-plane that have a finite number of boundary singularities, not
assumed a priori to be conical, and constant geodesic curvature
along each boundary arc.

\end{abstract}

\section{Introduction}

The Liouville equation $\lap v + 2 K e^v=0$ has a natural geometric
Neumann problem attached to it, that comes from the following
question:

Let $\Omega\subset \R^2$ be a domain with smooth boundary $\parc
\Omega$. \emph{What are the conformal Riemannian metrics on $\Omega$
having constant curvature $K$, and constant geodesic curvature along
each boundary component of $\parc \Omega$?} Here, we assume that the
metric extends smoothly to the boundary $\parc \Omega$.

An important property of the Liouville equation is that it is
conformally invariant. Thus, it is not very restrictive to consider
only simple symmetric domains $\Omega$, such as disks, half-planes
or annuli. The most studied case is when $\Omega=\R_+^2$. In that
situation, we are led to the Neumann problem

\begin{equation}\label{lio}
\left\{
\begin{array}{lll}
\Delta v+2 K e^v=0&\qquad\qquad&\text{in }\r^2_+=\{(s,t)\in\r^2:\ t>0\},\\[3mm]
{\displaystyle \frac{\partial v}{\partial t}}=c\, e^{v/2}&&\text{on }\partial\r^2_+,\\[3mm]
\end{array}
\right.\end{equation} where the first equation tells that the
conformal metric $e^v |dz|^2$ has constant curvature $K$, and the
free boundary condition gives that $\parc \R_+^2$ has constant
geodesic curvature $-c/2$ for that metric. The above problem was
fully solved by Zhang \cite{Zha} (in the finite energy case) and
Gálvez-Mira \cite{GaMi} (in general), as an extension of previous
results in \cite{LiZh,Ou} (see also \cite{ChLi,ChWa,HaWa}).

Recently, there has been some work on the geometric Neumann problem
for Liouville's equation in $\R_+^2$ in the presence of a boundary
singularity, i.e. the problem

$$
\left\{
\begin{array}{lll}
\Delta v+2 K e^v=0&\qquad\qquad&\text{in }\r^2_+=\{(s,t)\in\r^2:\ t>0\},\\[3mm]
{\displaystyle \frac{\partial v}{\partial t}}=c_1e^{v/2}&&\text{on }\partial\r^2_+\cap\{s>0\},\\[3mm]
{\displaystyle \frac{\partial v}{\partial t}}=c_2e^{v/2}&&\text{on
}\partial\r^2_+\cap\{s<0\},
\end{array}
\right.\qquad\qquad (P)
$$
with $K\in\{-1,0,1\}$ and $c_1,c_2\in\r$. In \cite{JWZ}, Jost, Wang
and Zhou gave a complete classification of the solutions to the
above problem under the following assumptions:

\begin{enumerate}
 \item
The metric $e^v |dz|^2$ has finite area in $\R_+^2$, i.e.

 \begin{equation}\label{finitehalf}
\int_{\R_+^2} e^v <\8.
 \end{equation}
 \item
The boundary $\parc \R_+^2$ has finite length for the metric $e^v
|dz|^2$, i.e.
 \begin{equation}\label{filen}
\int_{\R_-} e^{v/2} + \int_{\R_+} e^{v/2} <\8.
 \end{equation}
 \item
The metric $e^v|dz|^2$ has a \emph{boundary conical singularity} at
the origin, i.e. there exists $\lim_{z\to 0} |z|^{-2\alfa} e^v \neq
0$ for some $\alfa>-1$.
  \item
$K=1$.
\end{enumerate}
With these hypotheses, they showed that any solution to $(P)$
corresponds to the conformal metric associated to the sector of a
sphere of radius one limited by two circles that intersect at
exactly two points, or to the complement of a closed arc of circle
in the sphere, possibly composed with an adequate branched covering
of the Riemann sphere $\overline{\C}$. In particular, they provided
explicit analytic expressions for all these solutions.

Our main objective in this paper is to provide several improvements of the Jost-Wang-Zhou theorem. These are included in
Theorem \ref{t1}, Theorem \ref{c1}, Theorem \ref{mainth1} and
Corollary \ref{cordos}.

In Theorem \ref{mainth1} we will remove the last three hypotheses of
the above list in the Jost-Wang-Zhou result, and prove that any
solution to $(P)$ of finite area is a \emph{canonical solution}.
These canonical solutions have explicit analytic expressions and a
simple geometric interpretation as the conformal factor associated
to basic regions of $2$-dimensional space forms, up to composition
with suitable branched coverings of $\overline{\C}$ if $K=1$ (see
Section 2). For the case $K=1$ we recover the solutions obtained in
\cite{JWZ}, together with some new solutions corresponding to the
case that the boundary singularity at the origin is not conical; we
do not prescribe here any asymptotic behavior at the origin, nor the
finite length condition \eqref{filen}.

In Theorem \ref{t1} we give a general classification of all the
solutions to $(P)$, without any integral finiteness assumptions, in
the spirit of \cite{GaMi}. We show that the class of solutions to
$(P)$ is extremely large, but still it can be described in terms of
entire holomorphic functions satisfying some adequate properties. As
a matter of fact, we give such a result not only in $\R_+^2$ but
also in an arbitrary half-disk $D_{\ep}^+\subset \R_+^2$. That is,
we also give a general classification result for the solutions $v\in
C^2(\overline{D_{\ep}^+}\setminus\{0\})$ to the local problem

$$\left\{
\begin{array}{lll}
\Delta v+2 K e^v=0&\qquad&\text{in }D_{\ep}^+=\{(s,t)\in\r^2:\ s^2+t^2<\ep^2,\ t>0\},\\[3mm]
{\displaystyle \frac{\partial v}{\partial t}}=c_1e^{v/2}&&\text{on }I_{\ep}^+=\{(s,0)\in\r^2:\ 0<s<\ep\},\\[3mm]
{\displaystyle \frac{\partial v}{\partial t}}=c_2e^{v/2}&&\text{on
}I_{\ep}^-=\{(s,0)\in\r^2:\ -\ep<s<0\}.
\end{array}
\right.\qquad (L)
$$

In Theorem \ref{c1} we classify the solutions to the local problem
$(L)$ that satisfy the finite area condition

\begin{equation}\label{localfinit}
 \int_{D_{\ep}^+} e^v <\8,
\end{equation}
and give a general procedure to construct all of them. In
particular, we describe the asymptotic behaviour at the origin of
any solution to $(L)$ that satisfies \eqref{localfinit}. This is a
generalization to the case of boundary singularities of the
well-known results in \cite{Bry,ChWa,Hei,Nit,War} which describe the
asymptotic behaviour of metrics of constant curvature and finite
area in the punctured disk $\D^*$.

In Corollary \ref{cordos} we extend Theorem \ref{mainth1} to the
case of an arbitrary number of boundary singularities. This solves a
problem posed in \cite{JWZ}, under milder hypotheses. The basic
examples of conformal metrics of constant curvature with boundary
singularities and constant geodesic curvature along each boundary
component are the ones determined by circular polygons in
$\overline{\C}$, but there are many others. To obtain this larger
family, we consider \emph{immersed} circular polygons for which we
allow self-intersections, and give a differential-topological
criterion (Alexandrov embeddedness) for them to generate such
metrics, see Definition \ref{circo}. With this, Corollary
\ref{cordos} proves the converse: any conformal metric of finite
area and constant curvature on $\R_+^2$ (or equivalently on the unit
disk $\D$), with finitely many boundary singularities and constant
geodesic curvature along each boundary component, is one of those
\emph{circular polygonal metrics} constructed from
Alexandrov-embedded, possibly self-intersecting, circular polygons.
Analytically, those metrics will not have simple explicit
expressions; yet, one can still give some analytic information about
them. In Corollary \ref{corfor} we will describe for $K=1$ the
moduli space of these metrics, by parametrizing it in terms of their
associated Schwarzian maps, which have simple explicit expressions.

We have organized the paper as follows. In Section 2 we will present
the \emph{canonical solutions}, together with their geometric
interpretation and their basic properties. Section 3 contains some
preliminaries. In Section 4 we will study the local problem $(L)$ at
a boundary singularity, and prove Theorem \ref{t1}. In Section 5 we
will prove Theorem \ref{c1}, which describes all the solutions to
$(L)$ that satisfy the finite energy condition \eqref{localfinit}.
In Section 6 we will prove Theorem \ref{mainth1}, which states that
any finite area solution to $(P)$ is a canonical solution. In
Section 7 we will prove Corollary \ref{cordos} and Corollary
\ref{corfor} on the classification of conformal metrics of constant
curvature with a finite number of boundary singularities.

\section{The canonical solutions}
Our objective in this section is to describe, both analytically and
geometrically, an explicit family of solutions to $(P)$ satisfying
the finite energy condition \eqref{finitehalf}. We will prove in
Theorem \ref{mainth1} that these are actually \emph{all} the finite
energy solutions to $(P)$.

In all that follows, we assume that $K\in\{-1,0,1\}$, without loss
of generality.

\subsection{Analytic description}

\begin{definicion}\label{defican}
A \emph{canonical solution} is a function of one of the following
types:

\begin{enumerate}
  \item
$v_1:\R_+^2\flecha \R$ given by
\begin{equation}\label{eqv1}
v_1=\log\frac{4\lambda^2\gamma^2|z|^{2(\gamma-1)}}{(K\lambda^2+|z^{\gamma}-z_0|^2)^2}
\end{equation}
where $\gamma, \lambda>0$ and $z_0\in\c$ satisfy
$K\lambda^2+|z^{\gamma}-z_0|^2\neq0$ for all
$z\in\overline{\c^+}\equiv\overline{\r^2_+}$.
 \item
$v_2:\R_+^2\flecha \R$ given by
\begin{equation}\label{eqv2}
v_2=\log\frac{4\,\lambda^2}{|z|^2(K\lambda^2+|\log z-z_0|^{2})^2},
\end{equation}
where $\gamma, \lambda>0$ and $z_0\in\c$, satisfy $K\lambda^2+|\log
z-z_0|^2\neq0$ for all $z\in\overline{\c^+}\equiv\overline{\r^2_+}$.
Here, $\log z=\ln|z|+i \arg(z)$, where $\arg(z)\in[0,\pi]$.
\end{enumerate}
\end{definicion}

Let us observe some elementary properties of these \emph{canonical
solutions}, and explain for what choices of the constants $\gamma,
\lambda, z_0$ and $K$ they exist.

The function $v_1$ given by \eqref{eqv1} is well defined in
$\overline{\r^2_+}\backslash\{0\}$ if $K=1$ for all
$\gamma,\lambda>0$ , $z_0\in\c$. However, if $K=0,-1$, $v_1$ is well
defined if and only if $K\lambda^2+|z^{\gamma}-z_0|^2\neq0$. In
other words, if and only if the distance from the point $z_0$ to the
sector $\{z^{\gamma}:\ z\in\overline{\r^2_+}\backslash\{0\}\}$ is
bigger than $-K\lambda^2$. A simple analysis shows that this
happens:
\begin{itemize}
\item for $K=0$ if and only if $z_0=0$, or $z_0\neq 0$ and $\pi\gamma<\theta_0$ with $\theta_0=\arg (z_0)\in[0,2\pi)$.
\item for $K=-1$ if and only if $\lambda\leq|z_0|$, $\pi\gamma<\theta_0-\alpha_0$, and $|{\rm Im}(z)|>\lambda$ when ${\rm Re}(z)>0$. Here, $\theta_0=\arg (z_0)\in[0,2\pi)$ and $\alpha_0\in(0,\pi/2)$ with $\sin\alpha_0=\lambda/|z_0|$.
\end{itemize}

Besides, if $K=1$ the function $v_1$ satisfies the finite area
condition
$$
\int_{\r^2_+}e^{v_1}=\int_{\r^2_+}\frac{4\lambda^2\gamma^2|z|^{2(\gamma-1)}}{(K\lambda^2+|z^{\gamma}-z_0|^2)^2}<\infty
$$
for every $\gamma,\lambda,z_0$.

In the other cases, if it holds $K\lambda^2+|z^{\gamma}-z_0|^2\neq0$
for all $z\in\overline{\c^+}\equiv\overline{\r^2_+}$ (and not just
for all $z\in \overline{\R_+^2}\setminus\{0\}$), then the metric
trivially has finite area. Otherwise, it means that $z_0=0$ if
$K=0$, or $|z_0|=\landa$ when $K=-1$. But in these cases we clearly
have infinite area at the origin.

As a consequence, $v_1$ is a well defined function in
$\overline{\r^2_+}\backslash\{0\}$ with finite area if and only if
$K\lambda^2+|z^{\gamma}-z_0|^2\neq0$ for all
$z\in\overline{\c^+}\equiv\overline{\r^2_+}$, which is the condition
of Definition \ref{defican}. Observe that $\gamma<2$ when $K=0,-1$.

Analogously the function $v_2$ given in \eqref{eqv2} is well defined
in $\overline{\r^2_+}\backslash\{0\}$ and has finite area if and
only if $K\lambda^2+|\log z-z_0|^2\neq0$ for all
$z\in\overline{\c^+}\equiv\overline{\r^2_+}$.

In particular, if $K=1$ the condition $K\lambda^2+|\log
z-z_0|^2\neq0$ for all $z\in\overline{\c^+}$ is satisfied for every
$\lambda,z_0$. However, in the other cases, we need to impose that
the distance from the point $z_0$ to the strip $\{\log z:\
z\in\overline{\r^2_+}\backslash\{0\}\}=\{\zeta\in \C: 0<{\rm Im}
\zeta <\pi\}$ is bigger that $-K\lambda^2$. This condition happens
\begin{itemize}
\item for $K=0$ if and only if ${\rm Im}(z_0)<0$ or ${\rm Im}(z_0)>\pi.$
\item for $K=-1$ if and only if ${\rm Im}(z_0)<-\lambda$ or ${\rm Im}(z_0)>\pi+\lambda.$
\end{itemize}

This analysis together with a simple computation shows that these
canonical solutions are indeed finite area solutions to problem
$(P)$.

\begin{lema}\label{lemacan}
Any canonical solution $v:\overline{\R_+^2}\setminus \{0\}\flecha
\R$ is a solution to the geometric Neumann problem $(P)$ satisfying
$$\int_{\R_+^2} e^v <\8,$$ where the constants $c_1,c_2$ associated
to the problem are given by the following expressions in terms of
$\gamma$, $\lambda$ and $z_0:=r_0 e^{i\theta_0}$:
 \begin{enumerate}
 \item
For $v_1$ as in \eqref{eqv1}, \beq\label{C1}
c_1=2\,\frac{r_0}{\lambda}\,\sin\theta_0,\qquad
c_2=-2\,\frac{r_0}{\lambda}\,\sin(\theta_0-\pi\gamma). \eeq
 \item
For $v_2$ as in \eqref{eqv2}, \beq\label{C2}
 c_1=\frac{2}{\lambda}\,{\rm Im}(z_0),\qquad c_2=\frac{2}{\lambda}\,(\pi-{\rm Im}(z_0)).
\eeq
 \end{enumerate}
\end{lema}
\subsection{Geometric description}\label{dosdos}

Let $\cQ^2 (K)$ denote the $2$-dimensional space form of constant
curvature $K\in \{-1,0,1\}$, which will be viewed as
$(\Sigma_K,ds_K^2)$ where
$$\Sigma_K= \left\{\def\arraystretch{1.2}\begin{array}{ccc}
\overline{\C} & \text{ if } K=1,\\ \C & \text{ if } K=0, \\
\D\subset \C & \text{ if } K=-1,\end{array} \right.$$ and $ds_K^2$
is the Riemannian metric on $\Sigma_K$ given by
 \begin{equation}\label{metrik}
 ds_K^2=\frac{4 |d\zeta|^2}{(1+K|\zeta|^2)^2}.
 \end{equation}
So, a regular curve in $\Sigma_K$ has constant curvature if and only
if its image is a piece of a circle in $\bar{\C}$.

\begin{definicion}
Let $\cC_1,\cC_2$ be two different circles in $\overline{\C}$ such
that $\cC_1\cap \cC_2\neq \emptyset$, and let $\cU\subset
\overline{\C}$ be any of the regions in which $\cC_1$ and $\cC_2$
divide $\overline{\C}$. Assume that $\overline{\cU}$ is contained in
$\Sigma_K$. Then $\cU$ is called a \emph{basic domain} of
$\cQ^2(K)$.
\end{definicion}

Let now $\cU\subset \Sigma_K$ be a basic domain equipped with the
metric $ds_K^2$ in \eqref{metrik}. Note that one can conformally
parametrize $\cU$ by a biholomorphism $g:\C^+\flecha \cU$ such that
$g(\8)$ is a point $p\in \cC_1\cap \cC_2$, and in the case that
$\cC_1\cap \cC_2$ is not a single point $g(0)$ is also some $q\in
\cC_1\cap \cC_2$.

It is then clear from this process that the pull-back metric $g^*
(ds_K^2)$ produces a conformal metric of constant curvature $K$ in
$\overline{\C^+}\equiv \R_+^2$, which has constant geodesic
curvature along $\R_-$ and $\R_+$, and a singularity at the origin.
Also, this metric trivially has finite area, so we have a solution
to $(P)$ that satisfies \eqref{finitehalf}.

A similar process can be done if $K=1$, by considering $\cU$ to be
the complement of an arc of a circle in $\overline{\c}$. This would
correspond in some sense to taking $\cC_1=\cC_2$ in the above
process.

Furthermore, if $K=1$ and $\cC_1\cap \cC_2=\{p,q\}$ consists of two
points, we can easily create other finite area solutions to $(P)$,
starting from the basic region $\cU\subset \overline{\C}$. For that,
it suffices to consider a finite-folded branched holomorphic
covering of $\overline{\C}$, with branching points at $p$ and $q$.
If we denote this branched covering by $\Phi$, and consider
$\hat{g}:=\Phi \circ g$, the pullback metric of $ds_K^2$ via
$\hat{g}$ again describes as before a finite area solution to $(P)$.

This construction provides a geometric interpretation of the
canonical solutions. Indeed, we have:

\vspace{0.3cm}

\noindent {\bf Fact:} Let $v\in
C^2(\overline{\R_+^2}\setminus\{0\})$ be a canonical solution. Then
$e^v |dz|^2$ is the pullback metric on $\R_+^2$ of either:
 \begin{enumerate}
   \item[a)]
some basic region $\cU$ in $\cQ^2(K)$, or
 \item[b)]
the complement in $\cQ^2(1)\equiv \overline{\c}$ of a closed arc of
a circle,
 \end{enumerate}
possibly composed with a suitable branched covering of
$\overline{\C}$ in case $K=1$.

\vspace{0.3cm}

We do not give a direct proof of this fact, since it will become
evident from the proof of our main results. See Section
\ref{secgeom}.

\section{Preliminaries}

Let us start by explaining the classical relationship between the
Liouville equation and complex analysis. From now on we will
identify $\R^2$ and $\C$, and write $z=s+it \equiv (s,t)$ for points
in the domain of a solution to the Liouville equation.

 \begin{teorema}\label{repth}
Let $v:\Omega\subset \R^2\equiv \C \flecha \R$ denote a solution to
$\Delta u + 2K e^v=0$ in a simply connected domain $\Omega$. Then
there exists a locally univalent meromorphic function $g$
(holomorphic with $1 + K |g|^2>0 $ if $K\leq 0$) in $\Omega$ such
that
 \begin{equation}\label{rep}
 v = \log \frac{4 |g'|^2}{(1+ K |g|^2)^2}.
 \end{equation}
Conversely, if $g$ is a locally univalent meromorphic function
(holomorphic with $1 + K |g|^2>0 $ if $K\leq 0$) in $\Omega$, then
\eqref{rep} is a solution to $\Delta v + 2K  e^v =0$ in $\Omega$.
 \end{teorema}
Up to a dilation, we will assume from now on that $K\in \{-1,0,1\}$.
Also, observe that the function $g$ in the above theorem, which is
called the \emph{developing map} of the solution, is unique up to a
Möbius transformation of the form
 \begin{equation}\label{isom}
  g\mapsto \frac{ \alfa g - \bar{\beta}}{ K \beta g + \bar{\alfa}}, \hspace{1cm} |\alfa|^2 +K |\beta|^2 =
  1.
 \end{equation}

\begin{nota}\label{remiso}
The developing map $g$ has a natural geometric interpretation: if
$v\in C^2 (\Omega)$ is a solution to $\lap v + 2K e^v=0$, then its
developing map $g:\Omega\subseteq \C\flecha \Sigma\subseteq
\bar{\C}$ provides a local isometry from $(\Omega, e^v|dz|^2)$ to
$\cQ^2 (K)\equiv (\Sigma_K,ds_K^2)$, where $ds_K^2$ is given by
\eqref{metrik}.
\end{nota}

There is another relevant holomorphic function attached to any
solution $v$ of the Liouville equation. We will denote it by $Q$,
and it is given by the formulas below, where $g$ is the developing
map of $v$:
\begin{equation}\label{schw}
 Q:=v_{zz} - \frac{1}{2} \, v_z^2 = \{g,z\} := \left(\frac{g_{zz}}{g_z}\right)_{\!\!z} -
 \frac{1}{2}\left(\frac{g_{zz}}{g_z}\right)^2.
 \end{equation}Here, by definition $v_z=(v_s-iv_t)/2$ (and $g_z=g'$), and $\{g,z\}$ is
the classical \emph{Schwarzian  derivative} of the meromorphic
function $g$ with respect to $z$. Observe that $Q$ is holomorphic,
i.e. it does not have poles, and it does not depend on the choice of
the developing map $g$. We will call it the \emph{Schwarzian map}
associated to the solution $v$.

The following lemma gives some basic local properties of a solution
to the geometric Neumann problem for the Liouville equation along
the boundary. It is a consequence of some arguments in \cite{GaMi},
but we give a brief proof here for the convenience of the reader.

\begin{lema}\label{mainlem}
Let $D_{\varepsilon}^+ =\{z\in \C: |z|<\varepsilon, {\rm Im} z
>0\}$, and let $v\in C^2 (\overline{D_{\varepsilon}^+})$ be a
solution to $$\left\{\def\arraystretch{1.4}\begin{array}{lll} \lap v
+ 2 K v =0 & in & D_{\varepsilon}^+ \\ \displaystyle \frac{\parc
v}{\parc t} = c e^{v/2} & on & I_{\varepsilon} =(-\ep,\ep)\subset
\R\end{array} \right.$$ Then:
 \begin{enumerate}
   \item[(i)]
The Schwarzian derivative map $Q$ of $v$, defined by \eqref{schw},
takes real values along $I_{\ep}$, and extends holomorphically to
the whole disk $D_{\ep}$ by $Q(\bar{z})=\overline{Q(z)}.$
 \item[(ii)]
The developing map $g$ of $v$ can be extended to $D_{\ep}$ as a
locally univalent meromorphic function.
 \item[(iii)]
$g(s,0):I_{\ep}\flecha \overline{\C}$ is a regular parametrization
of a piece of a circle $\cC$ in $\overline{\C}$.
 \end{enumerate}
\end{lema}
\begin{proof}
By the Neumann condition $v_t=c e^{v/2}$ along $I_{\ep}$, we have
 \begin{equation*}
 {\rm Im}\,  Q
(s,0)= -\frac{1}{2} \left( \frac{c}{2} v'(s) e^{v(s)/2} -
\frac{c}{2} v'(s) e^{v(s)/2}\right)=0,
 \end{equation*}
for every $s\in I_{\ep}$. Thus, $(i)$ holds immediately by
Schwarzian reflection.

For $(ii)$, we only need to recall that if $q(z)$ is a holomorphic
function in a simply connected domain, then the equation $\{g,z\} =
q (z)$ always has a locally univalent meromorphic solution $g$,
which is unique up to linear fractional transformations. In our
case, we have $\{g,z\} =Q$ on $D_{\ep}^+$, and so $(ii)$ follows
from $(i)$.

Finally, $(iii)$ is clear from the fact that the developing map $g$
defines a local isometry from $(\overline{D_{\ep}^+},e^v|dz|^2)$
into $\cQ^2(K)$, and $I_{\ep}$ has constant curvature $-c/2$ for the
metric $e^v |dz|^2$, by the Neumann condition $v_t=ce^{v/2}$.
\end{proof}

For the proof of Theorem \ref{t1}, we will also need the following
elementary lemma.

\begin{lema}\label{l1}

Let $\widetilde{\Omega}=\{w\in\c:\ a<Re(w)<b\}$, with $-\infty\leq
a<b\leq+\infty$, and let $h:\widetilde{\Omega}\fl\overline{\c}$ be a
function such that $h(w+2\pi i)=h(w)$. Then, there exists a well
defined function $f:\Omega\fl\c$ on the topological annulus
$\Omega=\{z\in\c:\ a<\log|z|<b\}$ such that $h(w)=f(e^w)$ for all
$w\in\widetilde{\Omega}$.

Moreover, if $h$ is a meromorphic function then so it is $f$.
\end{lema}

\section{The local problem: proof of Theorem \ref{t1}.}\label{s3}

A general description of all solutions to $(L)$, in the spirit of
the main result in \cite{GaMi}, is given by the following theorem.
We let $D_{\ep}^*$ denote $\{z\in \C: 0<|z|<\ep$\}.
\begin{teorema}\label{t1}
Let $v$ be a solution of $(L)$. Then there exists a meromorphic
function $F:D_{\ep}^{\ast}\fl\overline{\c}$ such that $v$ can be
computed from \eqref{rep} for a locally univalent meromorphic
function $g:\overline{D_{\ep}^+}\setminus\{0\}\flecha \overline{\C}$
given by one of the following expressions:
\begin{enumerate}
\item[(i)] $g(z)=\psi(z^{\gamma}\,F(z))$, with $\gamma\in[0,1)$ and $F(r)\in\r\cup\{\infty\}$ for any $r\in\r\cap D_{\ep}^{\ast}$,
\item[(ii)] $g(z)=\psi(F(z)+\log(z))$, with $F(r)\in\r\cup\{\infty\}$ for any $r\in\r\cap D_{\ep}^{\ast}$,
\item[(iii)] $g(z)=\psi(z^{i \gamma}\,F(z))$, with $\gamma<0$ and $|F(r)|=1$ for any $r\in\r\cap D_{\ep}^{\ast}$.
\end{enumerate}
Here, $\psi$ is a M\"obius transformation and $g$ is holomorphic
with $1+K|g|^2>0$ if $K\leq 0$.

Conversely, let $g:\overline{D_{\ep}^+}\setminus\{0\}\flecha
\overline{\C}$ be a locally univalent meromorphic function,
holomorphic with $1+K|g|^2>0$ if $K\leq 0$, constructed from a
meromorphic function $F:D_{\ep}^*\flecha \overline{\C}$ as in
$(i)-(iii)$ above. Then, the function $v$ given by \eqref{rep} is a
solution of problem $(L)$.
\end{teorema}

\begin{nota} Theorem \ref{t1} also provides all the
solutions of the global problem $(P)$. For that, it is enough to
consider $\varepsilon=\8$ in the previous theorem, that is, to
change $D_{\varepsilon}^{\ast}$ by $\c^{\ast}$.
\end{nota}

\begin{proof}

Let $v\in C^2(\overline{D_{\ep}^+}\setminus\{0\})$ be a solution of
problem $(L)$, and consider an associated developing map $g$. As
explained in Lemma \ref{mainlem}, the Schwarzian map $Q$ of $v$,
given by \eqref{schw}, extends holomorphically to the punctured disk
$D_{\ep}^*$. Consider now the covering map $w\mapsto e^w$, from
$\widetilde{D_{\ep}^*}=\{z\in\c:\ Re(z)<\log\ep\}$ to $D_{\ep}^*$,
which is a local biholomorphism. Then, in the region of
$\widetilde{D_{\ep}^*}$ such that $0<{\rm Im} w<\pi$ we can take the
meromorphic map $\gt$ given by
\begin{equation}\label{1}
\gt(w)=g(e^w).
\end{equation}
Moreover, the Schwarzian of $\gt(w)$ satisfies
 \begin{equation}\label{scq}
\{\gt,w\}= e^{2w} Q(e^w) -\frac{1}{2}.
 \end{equation}
As $Q(e^w)$ is globally defined and holomorphic in
$\widetilde{D_{\ep}^*}$, we see by the existence of solutions to the
Schwarzian equation that $\gt(w)$ can be extended to a locally
univalent meromorphic function globally defined on
$\widetilde{D_{\ep}^*}$. In addition, since the right hand side of
\eqref{scq} is $2\pi i$-periodic, and since solutions to the
Schwarzian equation $\{y,w\}=q(w)$ are unique up to Möbius
transformations, we see that the meromorphic function
$\gt:\widetilde{D_{\ep}^*}\flecha \overline{\C}$ satisfies
\begin{equation}\label{2}
\gt(w+2\pi i)=\psi(\gt(w))
\end{equation}
for a certain M\"obius transformation $\psi$.

As explained in Lemma \ref{mainlem}, $\gt(w)$ lies on a circle
${\cal C}_1\subset \overline{\C}$ for $\{w\in
\widetilde{D_{\ep}}:\Im(w)=0\}$, and $\gt(w)$ lies on another circle
${\cal C}_2\subset \overline{\C}$ for $\{w\in
\widetilde{D_{\ep}}:\Im(w)=\pi\}$. We will study the behavior of $g$
in terms of the relative position of both circles.

\vspace{0.3cm}

\noindent{\bf Case 1: ${\cal C}_1$ intersects ${\cal C}_2$ in two
points or they coincide.}

\vspace{0.3cm}

If ${\cal C}_1$ and ${\cal C}_2$ share at least two points, then we
can consider a M\"obius transformation $\varphi$ such that
$\varphi({\cal C}_1)$ is the circle
$\r\cup\{\infty\}\subseteq\overline{\c}$ and $\varphi({\cal C}_2)$
is the circle given by a straight line passing through the origin
and $\infty\in\overline{\c}$. For that, observe that $\varphi$ is
the composition of a M\"obius transformation which maps the previous
two points of ${\cal C}_1\cap{\cal C}_2$ into $\{0,\infty\}$, and a
rotation with respect to the origin.

From (\ref{1}) and (\ref{2}), the new locally univalent meromorphic
maps $G=\varphi\circ g$ and $\widetilde{G}=\varphi\circ
\widetilde{g}$ satisfy
\begin{equation}\label{3}
\Gt(w)=G(e^w)
\end{equation}
and
\begin{equation}\label{4}
\Gt(w+2\pi i)=\Psi(\Gt(w))
\end{equation}
for a certain M\"obius transformation $\Psi$.

For any real number $r\in(-\infty,\log \ep)$ we have
$\Gt(r)\in\varphi({\cal C}_1)\subseteq\r\cup\{\infty\}$. Hence, by
the Schwarz reflection principle,
\begin{equation}\label{5}
\Gt(w)=\overline{\Gt(\overline{w})}, \qquad \text{ for all
}w\in\widetilde{D_{\ep}}.
\end{equation}
Thus, from (\ref{4}) and (\ref{5}),
\begin{equation}\label{sietest}
\Gt(r+\pi i)=\Psi(\Gt(r-\pi i))=\Psi\left(\,\overline{\Gt(r+\pi
i)}\,\right), \qquad \text{ for all }r\in(-\infty,\log\ep).
\end{equation}
And, since the set $\{\Gt(r+\pi i):\ r\in(-\infty,\log\ep)\}$ lies
on the circle $\varphi({\cal C}_2)$ and has no empty interior in
$\varphi({\cal C}_2)$, then
$$
\zeta=\Psi(\overline{\zeta}),\qquad \text{ for all
}\zeta\in\varphi({\cal C}_2).
$$

But a M\"obius transformation is determined by the image of three
points, and $\varphi({\cal C}_2)$ passes through $0$ and $\8$. So,
if we take an arbitrary point $\zeta_0\in\varphi({\cal
C}_2)\backslash\{0,\infty\}$ we easily obtain that
$$
\Psi(\zeta)=\frac{\zeta_0}{\overline{\zeta_0}}\, \zeta,\qquad\text{
for all }\zeta\in\overline{\c}.
$$
Therefore, from (\ref{4}), we get
$$
\Gt(w+2\pi i)=e^{i\theta_0}\ \Gt(w),\qquad w\in\widetilde{D_{\ep}},
$$
where $e^{i\theta_0}=\zeta_0/\overline{\zeta_0}$ for a real constant
$\theta_0\in[0,2\pi)$. Finally, in order to obtain $\Gt$ we observe
that the new meromorphic function
\begin{equation}\label{6}
H(w)=e^{-\frac{\theta_0}{2\pi}\,w}\ \Gt(w)
\end{equation}
satisfies
$$
H(w+2\pi i)=H(w), \qquad w\in \widetilde{D_{\ep}}.
$$
So, from Lemma \ref{l1}, there exists a well defined meromorphic
function $F(z)$ in the punctured disk $D_{\ep}^{\ast}$ such that
$$
H(w)=F(e^w),\qquad w\in \widetilde{D_{\ep}}.
$$
Hence, (\ref{5}) and (\ref{6}) give
\begin{equation}\label{7}
\Gt(w)=e^{\gamma w}\ F(e^w),\qquad w\in \widetilde{D_{\ep}},
\end{equation}
with $\gamma=\theta_0/(2\pi)\in[0,1)$ and
$F(z)=\overline{F(\overline{z})}$, $z\in D_{\ep}^{\ast}$.

In particular, the developing map $g$ of any solution of the local
problem $(L)$ when ${\cal C}_1$ and ${\cal C}_2$ have at least two
common points is given, from (\ref{3}) and (\ref{7}), by
\begin{equation}\label{8}
g(z)=\frac{A\,z^{\gamma}\,F(z)+B}{C\,z^{\gamma}\,F(z)+D},
\end{equation}
for certain complex constants $A,B,C,D$, with $AD-BC=1$, which
determine the M\"obius transformation $\varphi^{-1}$.

\begin{nota}\label{r1}
If ${\cal C}_1={\cal C}_2$, then $\zeta_0\in\r$ and so $\gamma=0$.
\end{nota}

\vspace{0.3cm}

\noindent{\bf Case 2: ${\cal C}_1$ intersects ${\cal C}_2$ in a
unique point.}

\vspace{0.3cm}

Let $p_0$ be the common point of the circles ${\cal C}_1$ and ${\cal
C}_2$. Then we consider a M\"obius transformation $\varphi$ that
maps ${\cal C}_1$ to the circle
$\r\cup\{\infty\}\subseteq\overline{\c}$ and maps ${\cal C}_2$ to
the circle $\{z\in\c: \Im(z)=\pi\}\cup\{\infty\}$. For that, observe
that $\varphi$ can be seen as a M\"obius transformation mapping
${\cal C}_1$ into $\r\cup\{\infty\}$ which maps $p_0$ to $\infty$,
composed with a homothety.

As in the previous case, we define the new locally univalent
meromorphic maps $G=\varphi\circ g$ and $\widetilde{G}=\varphi\circ
\widetilde{g}$ which satisfy \eqref{3}, \eqref{4}, \eqref{5} and
\eqref{sietest}.

Since the set $\{\Gt(r+\pi i):\ r\in(-\infty,\log\ep)\}$ lies on the
circle $\{z\in\c:\ Im(z)=\pi\}\cup\{\infty\}$ and has no empty
interior there, then
$$
\zeta=\Psi(\overline{\zeta}),\qquad \text{for all } \zeta\in
\{z\in\c: \Im(z)=\pi\}\cup\{\infty\}.
$$
Therefore, $ \Psi(\zeta)=\zeta+2\pi i, $ and so, from (\ref{4}), $$
\Gt(w+2\pi i)=\Gt(w)+2\pi i,\qquad w\in\widetilde{D_{\ep}}.$$ Now,
the new meromorphic function $H(w)=\Gt(w)-w$ satisfies $ H(w+2\pi
i)=H(w) $ for all $w\in \widetilde{D_{\ep}}$. Hence, using Lemma
\ref{l1} for the meromorphic function $H(w)$, there exists a well
defined meromorphic function $F(z)$ in the punctured disk
$D_{\ep}^{\ast}$ such that
\begin{equation}\label{12}
\Gt(w)=F(e^w)+w,\qquad w\in \widetilde{D_{\ep}}.
\end{equation}
Moreover, from (\ref{5}), $F(z)=\overline{F(\overline{z})}$, $z\in
D_{\ep}^{\ast}$.

With all of this, the developing map $g$ of any solution of the
local problem $(L)$ when ${\cal C}_1$ and ${\cal C}_2$ have only one
common point is given, from (\ref{3}) and (\ref{12}), by
\begin{equation}\label{13}
g(z)=\frac{A\,(F(z)+\log z)+B}{C\,(F(z)+\log z)+D},
\end{equation}
for certain complex constants $A,B,C,D$, with $AD-BC=1$, which
determine the M\"obius transformation $\varphi^{-1}$.

\vspace{0.3cm}

\noindent{\bf Case 3: ${\cal C}_1$ does not intersect ${\cal C}_2$.}

\vspace{0.3cm}

In this case, it is well known that there exists a M\"obius
transformation $\varphi$ such that the image of the circles ${\cal
C}_1$ and ${\cal C}_2$ are the circles centered at the origin with
radii $1$ and $R>1$, respectively.

We start by considering the locally univalent meromorphic maps
$G=\varphi\circ g$ and $\widetilde{G}=\varphi\circ \widetilde{g}$,
which satisfy again (\ref{3}) and (\ref{4}) for a certain M\"obius
transformation $\Psi$.

Given a real number $r\in(-\infty,\log \ep)$ we have $|\Gt(r)|=1$.
So, from the Schwarz reflection principle
\begin{equation}\label{16}
\Gt(w)=\frac{1}{\,\overline{\Gt(\overline{w})}\,}, \qquad
w\in\widetilde{D_{\ep}}.
\end{equation}
In addition, from (\ref{4}),
$$
\Gt(r+\pi i)=\Psi(\Gt(r-\pi
i))=\Psi\left(\frac{1}{\,\overline{\Gt(r+\pi i)}\,}\right), \qquad
r\in(-\infty,\log\ep).
$$
Thus, proceeding as in the previous cases, we have
$$
\zeta=\Psi(\frac{1}{\,\overline{\zeta}\,}),\qquad \text{ for
}|\zeta|=R,
$$
that is, $\Psi(\frac{1}{R}\,e^{i\theta})=R\,e^{i\theta}$ for any
$\theta\in\r$.

Therefore, $ \Psi(\zeta)=R^2\zeta, $ and so, from (\ref{4}),
$$
\Gt(w+2\pi i)=R^2\,\Gt(w),\qquad w\in\widetilde{D_{\ep}}.
$$
Then we can apply Lemma \ref{l1} to the meromorphic function
$$
H(w)=e^{-\frac{\log(R^2)}{2\pi i}\,w}\ \Gt(w).
$$
Hence, there exists a well defined meromorphic function $F(z)$ in
$D_{\ep}^{\ast}$ such that
$$
\Gt(w)=e^{i\gamma w}F(e^w),\qquad w\in \widetilde{D_{\ep}}
$$
for the negative real constant $\gamma=-\frac{\log(R^2)}{2\,\pi}$.
Moreover, from (\ref{16}), $F(z)\,\overline{F(\overline{z})}=1$, for
any $z\in D_{\ep}^{\ast}$.

As a consequence, the developing map $g$ of any solution of the
local problem $(L)$ when ${\cal C}_1$ and ${\cal C}_2$ have no
common points is given, from (\ref{3}), by
\begin{equation}\label{17}
g(z)=\frac{A\,z^{\gamma\, i}F(z)+B}{C\,z^{\gamma\, i}F(z)+D},
\end{equation}
for certain complex constants $A,B,C,D$, with $AD-BC=1$, which
determine the M\"obius transformation $\varphi^{-1}$.

This completes the proof of the first part of Theorem \ref{t1}.
Also, the converse part of the theorem is just a straightforward
computation, so we are done.
\end{proof}

\section{Finite area: proof of Theorem \ref{c1}}

The following result describes the solution to problem $(L)$ under
the additional assumption \eqref{localfinit} of finite area.
\begin{teorema}\label{c1}
Let $v$ be a solution of $(L)$ that satisfies the finite energy
condition \eqref{localfinit}. Then, its developing map $g$ is given
by the cases (i) or (ii) in Theorem \ref{t1}, and $F$ does not have
an essential singularity at the origin.

In particular, $g$ can be continuously extended to the origin, and
the Schwarzian map $Q:D_{\ep}^{\ast}\fl\c$ of $v$ has at most a pole
of order two at $0$.
\end{teorema}

\begin{proof}
Let us start by explaining that it suffices to prove the result for
the case $K=1$. Indeed, let $v$ be a solution of
$(L)-$\eqref{localfinit} for a constant $K=K_0=-1,0$ and $g$ an
associated developing map. Now, let us consider the function $v_1$
given by \eqref{rep} for the map $g$ and $K=1$. Then, $v_1$ is also
a solution of $(L)$, but in this case for $K=1$, and it also
satisfies \eqref{localfinit} since
$$
\int
e^{v_1}\,|dz|=\int\frac{4|g'(z)|^2}{(1+|g(z)|^2)^2}\,|dz|\leq\int\frac{4|g'(z)|^2}{(1+K_0\,|g(z)|^2)^2}\,|dz|=\int
e^v\,|dz|<\infty.
$$
In other words, if the result is true for $K=1$, it will
automatically be true for $K=-1,0$, as claimed.

Thus, let $g$ be the developing map of a solution to $(L)$ for
$K=1$. First, let us prove that the lengths of the semicircles $$
C_r=\{z\in D_{\varepsilon}^{\ast}:\ |z|=r,\ {\rm Im}(z)\geq0\},
$$
for the metric $e^v|dz|^2$ of constant curvature $K=1$, tend to zero
when $r$ goes to zero.

If we denote by $L(r)$ the length of the semicircle $C_r$ for $e^v
|dz|^2$, and write $z=r\,e^{i\theta}$, we have from \eqref{rep}
$$
L(r)=r\int_0^{\pi} e^{v/2}d\theta=r\int_0^{\pi}
\frac{2|g'|}{1+|g|^2}d\theta\leq2\pi \sup_{0\leq\theta\leq\pi}(r\,
g^{\sharp}(r\,e^{i\theta}))
$$
where $g^{\sharp}(z)$ is the spherical derivative of $g$ with
respect to $z$, that is,
$$
g^{\sharp}(z):=\frac{|g'(z)|}{1+|g(z)|^2}.
$$

Hence,
\begin{equation}\label{18}
\limsup_{|z|\rightarrow0} L(|z|)\leq
2\pi\limsup_{z\rightarrow0}\left(|z|\,g^{\sharp}(z)\right).
\end{equation}

\vspace{0.3cm}

\noindent {\it Claim: in the above conditions, we have
$\limsup_{z\rightarrow0}\left(|z|\,g^{\sharp}(z)\right)=0$.}

\vspace{0.3cm}

Let us prove the claim above. Assume that
$\limsup_{z\rightarrow0}\left(|z|\,g^{\sharp}(z)\right)\neq0$, and
so $\limsup_{z\rightarrow0}g^{\sharp}(z)=\infty.$ Consider for some
$r\in (\ep/4,\ep/2)$ the fixed domain
$$\Omega=\{z\in\r^2_+:\ r\leq|z|\leq\varepsilon\},$$ and the family of functions
$$g_n(z)=g\left(\frac{z}{2^n}\right),\quad z\in\Omega.$$  That is,
the functions of the family $\mathfrak{G}=\{g_n\}_{n\in\n}$ are
nothing but the   function $g$ evaluated over a domain
$\Omega_n\subset\r^2_+$ that gets smaller and closer to the origin
as $n$ increases. Moreover, by the choice of $r$, it holds
$\Omega_n\cap\Omega_{n+1}\neq \emptyset$ and
$\Omega_n\cap\Omega_{n+2}=\emptyset$.  The Theorem of Marty (see
\cite{Ma}) characterizes a family of meromorphic functions $
\mathfrak{G}$ as a normal family if and only if for every compact
$K\subset\Omega$ there exists a constant $M(K)$ such that $g_n^{
\sharp}(z)\leq M(K)$ on $K$ for every function $g_n\in
\mathfrak{G}$. Thus, as
$\limsup_{z\rightarrow0}g^{\sharp}(z)=\infty$, $\mathfrak{G}$ is not
normal on $\Omega$. On the other hand, the Theorem of Montel  (see
\cite{Mo}) asserts that as $\mathfrak{G}$ is not normal there will
be at least one function $g_{n_0}\in  \mathfrak{G}$ that takes every
complex value with at most two exceptions on $\Omega$. Then, as
$\mathfrak{G}\setminus \{g_{n_0}\}$ is not a normal family, we can
iterate this argument and conclude that $g$ assumes every complex
value infinitely many times with at most to exceptions in a
neighborhood of the origin in $\r^2_+$. This means that $g(z)$
covers  an infinite area on $\{z\in D_{\varepsilon}^{\ast}:\ {\rm
Im}(z)\geq0\}$, which is a contradiction.

Thus, the claim is proved.

\hspace{0.2cm}

Observe that $g$ is a local isometry from $\{z\in
D_{\varepsilon}^{\ast}:\ {\rm Im}(z)\geq0\}$ with the metric
$e^v|dz|^2$ into the unit sphere $\overline{\c}$ with its standard
metric. Also, we know from Lemma \ref{mainlem} that
$g(D_{\varepsilon}\cap\r^+)$ lies on a circle ${\cal
C}_1\subseteq\overline{\c}$ and $g(D_{\varepsilon}\cap\r^-)$ lies on
a circle ${\cal C}_2\subseteq\overline{\c}$. So, $g(C_r)$ is a curve
in the sphere $\overline{\c}$ with length $L(r)$, joining the point
$g(r)$ of ${\cal C}_1$ and the point $g(-r)$ of ${\cal C}_2$.

Thus, the case $(iii)$ in Theorem \ref{t1} cannot happen because
this kind of solutions only occur when ${\cal C}_1\cap{\cal
C}_2=\emptyset$, which contradicts $\lim_{r\rightarrow0}L(r)=0$.

The solutions of type $(ii)$ in Theorem \ref{t1} happen when ${\cal
C}_1$ intersects ${\cal C}_2$ at a unique point $p_0\in
\overline{\c}$. Let us see that, in this case,
$\lim_{z\rightarrow0}g(z)=p_0$, what shows that the function $F(z)$
such that $g(z)=\psi(F(z)+\log(z))$ cannot have an essential
singularity at the origin as we wanted to show.

Let $\{z_n\}$ be a sequence converging to $0$, and $\delta>0$ small
enough. Now, let us prove that there exists $n_0\in \N$ such that if
$n\geq n_0$ then $d(g(z_n),p_0)<\delta$, where $d(,)$ denotes
distance in $\overline{\c}\equiv \cQ^2(1)$.

Let $D$ be the open disk of radius $\delta/2$ centered at $p_0$ and
$$\widehat{\delta}=\min\{d({\cal C}_1\backslash D,{\cal
C}_2),d({\cal C}_1,{\cal C}_2\backslash D)\}.$$ Since
$\lim_{r\rightarrow0}L(r)=0$ there exists $r_0>0$ such  that if
$r<r_0$ then $L(r)<\min \{ \delta,\widehat{\delta}\}$. Now, we
choose $n_0$ such that $|z_n|<r_0$ for all $n\geq n_0$. Thus,
$L(|z_n|)<\widehat{\delta}$ and so $g(|z_n|)\in{\cal C}_1\cap D$ and
$g(-|z_n|)\in{\cal C}_2\cap D$. Hence, if $d(g(z_n),p_0)\geq\delta$
we would have
$$
L(|z_n|)\geq
d(g(|z_n|),g(z_n))+d(g(z_n),g(-|z_n|))>\frac{\delta}{2}+\frac{\delta}{2}=\delta,
$$
which is a contradiction. This proves that
$\lim_{z\rightarrow0}g(z)=p_0$.

Let us show now that the Schwarzian derivative map $Q(z)$, which is
well defined in $D_{\varepsilon}^{\ast}$, has at most a pole of
order two at the origin. Since the Schwarzian derivative is
invariant under M\"obius transformations, in this case we only need
to do the computation for $g(z)=F(z)+\log z.$ As
$$
\frac{g''(z)}{g'(z)}=\frac{z^2 F''(z)-1}{z \left(z
   F'(z)+1\right)}
$$
has a pole of order one at the origin, we get from \eqref{schw} that
$Q(z)$ has at most a pole of order two there.

Finally, we analyze the solutions of the case $(i)$ in Theorem
\ref{t1}. They correspond to the situation in which ${\cal C}_1\cap
{\cal C}_2$ has exactly two points, or ${\cal C}_1={\cal C}_2$.

If ${\cal C}_1={\cal C}_2$ then from Remark \ref{r1} the associated
developing map $g$ is given by $g(z)=\psi(F(z))$, where $F(z)$ is a
meromorphic function in $D_{\varepsilon}^{\ast}$. In addition, the
meromorphic function $F(z)$ cannot have an essential singularity at
the origin. Otherwise, since $F(\overline{z})=\overline{F(z)}$,
$g(z)$ would take every value of $\overline{\c}$, except at most two
points, infinity many times in $\{z\in D_{\varepsilon}^{\ast}:\ {\rm
Im}(z)\geq 0 \}$ what would contradict the finite area condition.

If ${\cal C}_1\cap{\cal C}_2=\{p_1,p_2\}$, a similar argument to the
one we just used for the case that ${\cal C}_1\cap{\cal
C}_2=\{p_0\}\subset \overline{\C}$ lets us prove that there exists a
unique $i_0\in \{1,2\}$ such that $\lim_{z\rightarrow
0}g(z)=p_{i_0}$ . Thus, $g$ can be continuously extended to the
origin with $g(0)=p_{i_0}$ and $F(z)$ does not have an essential
singularity at $0$. Let us outline this argument.

Take again a sequence $\{z_n\}\to 0$ and $ 0<\delta$ satisfying
$\delta<d(p_1,p_2)/3$. Consider $D_i$ the open disk centered at
$p_i$ and radius $\delta/2$, and let
$$
\widehat{\delta}=\min\{d({\cal C}_1\backslash(D_1\cup D_2),{\cal
C}_2),d({\cal C}_1,{\cal C}_2\backslash(D_1\cup D_2))\}.
$$ Arguing as before, we can show that there exists a unique $i_0\in \{1,2\}$ such that $g(|z_n|),g(-|z_n|)\in
D_{i_0}$ for $n$ sufficiently large.

Hence, every point $z_n$ with $n$ sufficiently large satisfies
$d(g(z_n),p_{i_0})<\delta$, since otherwise
$d(g(z_n),p_{i_0})\geq\delta$, and so $L(|z_n|)\geq \delta$, which
contradicts that ${\rm lim}_{r\to 0} L(r)=0$. Therefore,
$\lim_{z\rightarrow0}g(z)=p_{i_0}$.

In order to finish the proof, let us show that the Schwarzian
derivative $Q(z)$ of $g(z)$ has at most a pole of order two at the
origin. Again by the invariance of the Schwarzian derivative under
M\"obius transformations, we only need to do the computation for
$g(z)=z^{\gamma}\,F(z)$. And since $F(z)$ has no essential
singularity at the origin,
$$
\frac{g''(z)}{g'(z)}=\frac{z^2  F''(z)+2 \gamma
  z F'(z)+(\gamma-1) \gamma F(z)}{z \left(z
   F'(z)+\gamma F(z)\right)}
$$
has at most a pole of order one there. Thus, from \eqref{schw},
$Q(z)$ has at most a pole of order two at the origin. This completes
the proof of Theorem \ref{c1}.
\end{proof}

Theorem \ref{c1} shows that for classifying the solutions to
$(L)$-\eqref{localfinit}, it suffices to determine when the
functions $v$ given by \eqref{rep} in terms of a developing map $g$
as in the statement of Theorem \ref{c1} verify \eqref{localfinit}.
We do this next.

First, assume that $g$ is given by case $(i)$ in Theorem \ref{t1},
where $F$ has at $0$ either a pole or a finite value. If we let
$A,B,C,D$ denote the coefficients of the Möbius transformations
defining $g$, with $AD-BC=1$, a computation gives $$e^v =\frac{
4|\gamma z^{\gamma-1} F(z) + z^{\gamma} F'(z)|^2}{(|C z^{\gamma}
F(z) + D |^2 + K |Az^{\gamma} F(z)+B|^2)^2}.$$ If $F$ has a finite
value (resp. a pole) at $0$, and if $|D|^2+K|B|^2\neq 0$ (resp.
$|C|^2 + K |A|^2\neq 0$), it is easy to check that near $0$ we have
$e^v = |z|^{2\alfa} \, a(z)$ where $\alfa>-1$ and $a(z)$ is
continuous at $0$, with $a(0)\neq 0$. Thus \eqref{localfinit} holds
in these cases.

We suppose now that $F$ has a finite value at $0$ and
$|D|^2+K|B|^2=0$. If $K=0$, then $D=0$ and
$e^v=a(z)|z|^{-2(\gamma+k+1)}$ with $a(z)$ continuous at $0$ and
$a(0)\neq0$, where $k$ is the order of the zero that $F$ has at the
origin (if $F(0)\neq 0$ then $k=0$). Thus \eqref{localfinit} does
not hold. If $K=-1$, the conditions $AD-BC=1$ and $|D|^2=|B|^2$ give
that $\delta=-A\overline{B}+C\overline{D} $ is such that
$|\delta|=1$, and so we can estimate
$$\def\arraystretch{2}\begin{array}{ll}
\displaystyle{\int_{D_{\varepsilon}^+}e^v}&=
\displaystyle{\int_{D_{\varepsilon}^+}\frac{4|\gamma z^{\gamma-1} F(z) + z^{\gamma} F'(z)|^2}{|z|^{2\gamma}|F|^2\left((|C|^2-|A|^2)|z|^{\gamma}|F(z)|+2{\rm Re}(\delta\frac{F(z)z^{\gamma}}{|F(z)||z|^{\gamma}})\right)^2}}\\
&\geq \displaystyle{\int_{D_{\varepsilon}^+}\frac{4|\gamma z^{\gamma-1} F(z) + z^{\gamma} F'(z)|^2}
{|z|^{2\gamma}|F|^2\left(|\, |C|^2-|A|^2 \, | \, |z|^{\gamma}|F(z)|+ 2\right)^2}}\\
& \geq
R\displaystyle{\int_0^{\varepsilon}\frac{dr}{r}=\infty},\end{array}$$
for a certain constant $R>0$, where $r=|z|$. So, \eqref{localfinit}
does not hold in this case.

We consider now the situation when $F$ has a pole of order $k\geq 1$
and $|C|^2+K|A|^2=0$. If $K=0$, we easily check that close to the
origin, $e^v=a(z)|z|^{2(\gamma-k-1)}$ with $a(z)$ continuous at $0$
and $a(0)\neq0$. So, \eqref{localfinit} does not hold since
$k>\gamma$. If $K=-1$ and $\delta:=-A\overline{B} + C\overline{D}$,
we also have $|\delta|=1$, and since $k>\gamma$ we can estimate as
before
 \begin{equation}\label{doubes}
 \begin{array}{ll}
\displaystyle{\int_{D_{\varepsilon}^+}e^v}&=\displaystyle{\int_{D_{\varepsilon}^+}\frac{4|\gamma z^{\gamma-1} F(z) +
z^{\gamma} F'(z)|^2}{ |F|^2|z|^{2\gamma}\left( \frac{|D|^2-|B|^2}{|F||z|^{\gamma}}+
2{\rm Re}(\delta \frac{F(z)z^{\gamma}}{|F||z|^{\gamma}}) \right)^2}}\\
&\geq
R\displaystyle{\int_0^{\varepsilon}\frac{dr}{r}=\infty},\end{array}
 \end{equation}
 for
a certain constant $R>0$, where $r=|z|$, and so \eqref{localfinit}
does not hold.

Next, assume that $g$ is given by case $(ii)$ in Theorem \ref{t1},
where again $F$ has at $0$ a pole or a finite value. Arguing as
before, since
 \begin{equation}\label{des2}
 e^v = \frac{4 |F'(z) + 1/z|^2}{(|C (F(z)+\log z)+D|^2
+ K |A (F+\log z) + B|^2)^2},
 \end{equation} we see that if $|C|^2+K|A|^2\neq 0$ and $F$ has a finite value (resp. a pole of order $k\geq 1$) at $0$,
then near $0$ we have $e^v = |z|^{-2} ({\rm ln} |z|)^{-4} a(z)$
(resp. $e^v=|z|^{2(k-1)} a(z)$) where $a(z)$ is continuous at $0$,
with $a(0)\neq 0$. Again, \eqref{localfinit} holds automatically.

Suppose now that $|C|^2 +K|A|^2=0$. If $K=0$, it is immediate to see
that \eqref{localfinit} does not hold. So, we are left with the case
$K=-1$ and $|A|=|C|$. Observe that this condition implies that
$g(0)\in \S^1$. As $g$ is given by case $(ii)$ in Theorem \ref{t1},
the images of $(0,\varepsilon)$ and $(-\varepsilon,0)$ by $g$ are two
circle arcs meeting tangentially at one point of $\S^1$.

Assume first that $F$ has a finite value at $0$.  If the constant
$\delta=-A\overline{B} + C\overline{D}$ is such that ${\rm Re}\,
\delta=0$, then we easily check that close to the origin
$$e^v=\frac{4|F'(z)+1/z|^2}{ (|D|^2-|B|^2-2{\rm Im}\delta ({\rm
Im}F(z)+\arg z))^2}\geq R |F'(z)+1/z|^2$$ for some $R>0$. Thus
\eqref{localfinit} does not hold. On the other hand, if ${\rm
Re}\delta\neq 0$ the asymptotic behavior of $e^v$ is
$$
e^v=\frac{|F'+1/z|^2}{(\ln|z|)^2(\frac{|D|^2-|B|^2}{2 \ln|z|}+{\rm
Re}\delta(1+\frac{{\rm Re}F(z)}{\ln|z|})-{\rm Im}\delta \frac{({\rm
Im}F(z)+\arg z)}{\ln|z|})^2}=\frac{a(z)}{|z|^2(\ln|z|)^2},$$ where
$a(z)$ is continuous at $0$ and $a(0)\neq0$. Then, we can easily
check that \eqref{localfinit} holds. Furthermore, a computation
shows that $c_1=-c_2 \in (-2,2)$.

Now, assume that $F$ has a pole at $0$, $K=-1$ and $|A|=|C|$. If
$\delta=-A\overline{B} + C\overline{D}$, we can proceed as we did
before and estimate
$$\begin{array}{ll}
\displaystyle{\int_{D_{\varepsilon}^+}e^v}&=\displaystyle{\int_{D_{\varepsilon}^+}\frac{4|F'(z)+1/z|^2}{ |F+\log z|^2\left( \frac{|D|^2-|B|^2}{|F+\log z|}+2{\rm Re}(\delta( \frac{F(z)+\log z}{|F+\log z|}) \right)^2}}\\
&\geq
R\displaystyle{\int_0^{\varepsilon}\frac{dr}{r}=\infty},\end{array}$$
for a certain constant $R>0$, where $r=|z|$. Again,
\eqref{localfinit} does not hold.

This completes the classification of all the solutions to
$(L)$-\eqref{localfinit}. In particular, we have the following
description of the asymptotic behavior of such solutions.

\begin{corolario}
Let $v\in C^2(\overline{D_{\ep}^+}\setminus\{0\})$ be a solution to
$(L)$-\eqref{localfinit}. There are three possible asymptotic
behaviors for $v$ at $0$:
 \begin{enumerate}
   \item
$\lim_{z\to 0} |z|^{-2\alfa} e^v \neq 0$, for some $\alfa>-1$, i.e.
$e^v |dz|^2$ has at $0$ a boundary conical singularity.
 \item
$\lim_{z\to 0} |z|^2 ({\rm ln } |z|)^4 e^v \neq 0$.
 \item
$\lim_{z\to 0} |z|^2 ({\rm ln } |z|)^2 e^v \neq 0$.
 \end{enumerate}
Here, the last case happens only when $K=-1$ and the boundary has
infinite length around $0$. In this last situation, it holds
$c_1=-c_2\in (-2,2)$.
\end{corolario}

\section{Global solutions: proof of Theorem \ref{mainth1}.}

In Section 2 we described in detail a large family of explicit
solutions to $(P)$ with finite area: the \emph{canonical solutions}.
We prove next that these are actually the only solutions to $(P)$
with finite area.

\begin{teorema}\label{mainth1}
Any solution to $(P)$ satisfying the finite energy condition
\eqref{finitehalf} is a canonical solution.
\end{teorema}
\begin{proof}
Let $v$ be a solution of $(P)$ such that
$$
\int_{\r^2_+}e^v\,dx\,dy<\infty,
$$
and let $Q(z)$ be the Schwarzian map associated to $v$. From Theorem
\ref{c1}, $Q(z)$ has at most a pole of order two at the origin.

We observe that the new meromorphic function $h(w)=g(z)$ with
$z=-1/w$ is also the developing map of a solution $\widehat{v}$ to
$(P)$. Thus, since
$$
\int_{\r^2_+}e^{\widehat{v}}\,dx\,dy=\int_{\r^2_+}e^v\,dx\,dy<\infty,
$$
we obtain again from Theorem \ref{c1} that the Schwarzian derivative
$\widehat{Q}(w)$ has at most a pole of order two at the origin.

By using that $w^4\, \widehat{Q}(w)=Q(z)$, we conclude that $Q(z)$
is a holomorphic function in $\c^{\ast}$ with at most a pole of
order two at the origin and at least a zero of order two at
infinity. Therefore,
$$
Q(z)=\frac{c}{z^2},\qquad z\in\c^{\ast}
$$
for a certain constant $c\in\c$.

The solutions $g(z)$ of the Schwarzian equation \eqref{schw} for
$Q(z)=c/z^2$ are well known. They are given by $g(z)=\psi(\log(z))$
if $2c=1$ and by $g(z)=\psi(z^\gamma)$ if $2c=1-\gamma^2\neq1$,
where $\psi$ is an arbitrary M\"obius transformation. In the latter
case, if follows from Theorem \ref{c1} that in our situation
$\gamma$ must be a real constant. In fact, up to composition with
the M\"obius transformation $z\rightarrow1/z$ if necessary, we can
assume $\gamma>0$.

Finally, in order to finish the proof, we compute the solutions $v$
to our problem depending of the value of $g(z)$.

Let
$$
g(z)=\psi(z^{\gamma})=\frac{A\,z^{\gamma}+B}{C\,z^{\gamma}+D}
$$
with $AD-BC=1$.

Then, from \eqref{rep},
$$
e^v=\frac{4\gamma^2|z|^{2(\gamma-1)}}{(K|B|^2+|D|^2+(KA\overline{B}+C\overline{D})z^{\gamma}+
(K\overline{A}B+\overline{C}D)\overline{z}^{\gamma}+(K|A|^2+|C|^2)|z|^{2\gamma})^2}\,.
$$
If $K|A|^2+|C|^2=0$, an argument as in \eqref{doubes} proves that
$\int_{\R_+^2} e^v <\8$.  On the other hand, if $K|A|^2+|C|^2\neq 0$
we can take
\begin{equation}\label{landa}
\lambda=\frac{1}{\left|\,K|A|^2+|C|^2\,\right|},\qquad
z_0=-\frac{K\overline{A}B+\overline{C}D}{K|A|^2+|C|^2}
\end{equation}
and so we have
$$
e^v=\frac{4\lambda^2\gamma^2|z|^{2(\gamma-1)}}{(K\lambda^2+|z^{\gamma}-z_0|^2)^2}\,
,
$$
that is, we obtain the canonical solution \eqref{eqv1}, as wished.

Now, let us consider the case
\begin{equation}\label{geloj}
g(z)=\psi(\log z)=\frac{A\,\log z+B}{C\,\log z+D}.
 \end{equation}
with $AD-BC=1$.

Then, from \eqref{rep},
\begin{eqnarray*}
e^v&=&4/(|z|^2(K|B|^2+|D|^2+(KA\overline{B}+C\overline{D})\log z+ (K\overline{A}B+\overline{C}D)\log \overline{z}+\\
&&+(K|A|^2+|C|^2)|\log z|^{2})^2).
\end{eqnarray*}
If $K|A|^2+|C|^2=0$, the function $v$ has infinite area in $\r^2_+$.
This follows directly from our discussion after Theorem \ref{c1}
except when ${\rm Re} \delta\neq 0$ and $K=-1$, where here
$\delta:=-A\overline{B} + C\overline{D}$. But in that case, using
that $|A|=|C|$, the condition that the map $g$ in \eqref{geloj} must
satisfy $|g(z)|<1$ for every $z\in \C^+$ leads to the inequality
$$2{\rm Re} \delta \ln |z| >|B|^2-|D|^2 + 2 {\rm Im} \delta \, {\rm
arg} z,$$ which cannot hold since $\ln |z|:\C^+\flecha \R$ is
surjective. This proves the claim.

If $K|A|^2+|C|^2\neq0$ the function $v$ can be rewritten as
$$
e^v=\frac{4\,\lambda^2}{|z|^2(K\lambda^2+|\log z-z_0|^{2})^2},
$$
where $\lambda$ and $z_0$ are chosen as in (\ref{landa}). This
completes the proof of Theorem \ref{mainth1}.
\end{proof}

As a consequence, we get:
 \begin{corolario}\label{maincor}
Given $K\in\{-1,0,1\}$, $c_1,c_2\in\r$, there exists a solution to
the problem $(P)$ with finite area if and only if
\begin{itemize}
\item $K=1$, or
\item $K=0$ and $c_i<0$ for some $i\in\{1,2\}$, or
\item $K=-1$ and one of these conditions are satisfied: $$c_1<-2,\qquad or \qquad c_2<-2, \qquad or\qquad c_1+c_2<0.$$
\end{itemize}
\end{corolario}
\begin{proof}
It suffices to prove the result for canonical solutions. Consider
first of all a canonical solution given by \eqref{eqv1} and write
$z_0=r_0e^{i\theta_0}$. We already explained in Section 2 the
restrictions between the parameters appearing in \eqref{eqv1} for
the solution to be well defined. Also, the relationship between the
constants $c_i$ and these parameters is given in Lemma
\ref{lemacan}. From this, we see directly that there are always
canonical solutions with $K=1$ for every $c_1,c_2\in \R$, of the
form \eqref{eqv1}.

In the cases $K=0,-1$, the situation is more restrictive. Label
$x=\theta_0$, $y=\theta_0-\pi\gamma$ and $R_0=2r_0/\lambda$. From
Lemma \ref{lemacan} we see that \beq\label{c1c2}
c_1=R_0\sin(x)\qquad c_2=-R_0\sin(y). \eeq By our analysis in
Section 2, we have: \begin{itemize}
\item if $K=-1$, then $0<\alpha_0<y<x<2\pi-\alpha_0$ where $\alfa_0\in (0,\pi/2)$ and $R_0=2/\sin(\alfa_0)>2$. So, if $\alpha_0<y\leq\pi/2$ (resp. $3\pi/2\leq x<2\pi-\alpha_0$) we get $c_2<-2$ (resp. $c_1<-2$),
and if $\pi/2<  y<x< 3\pi/2$ we have $c_1+c_2<0$. Conversely, assume
that $c_1$ and $c_2$ satisfy the restrictions above, choose $R_0>2$
such that $c_1/R_0, c_2/R_0\in [-1,1]$, and call
$\sin(\alpha_0)=2/R_0$. Then, it is clear that some choices of
$x=\arcsin(c_1/R_0)$ and $y= \arcsin(-c_2/R_0)$ satisfy
$\alpha_0<y<x<2\pi-\alpha_0$. So we can find $z_0$ and $\lambda$ as
in Section 2, such that \eqref{C1} holds.
 \item If $K=0$, then $0<y<x<2\pi$. Thus, a simple analysis shows that
$c_1$ and $c_2$ can not be positive simultaneously. The converse is
analogous to the previous $K=-1$ case.
   \end{itemize}
Finally, assume that the canonical solution is given by
\eqref{eqv2}. Again, there are no restrictions if $K=1$. On the
other hand, from \eqref{C2} in the cases $K=0$ and $K=-1$, the
condition $c_1+c_2>0$ must be satisfied. Moreover, because of the
restriction we had for this kind of solutions, at least one $c_i$
must be strictly negative (strictly less than $-2$ in the case
$K=-1$). The converse is trivial.
\end{proof}
\section{Uniqueness of polygonal circular metrics}\label{secgeom}

Let us start by stating the following result, which follows from the
proof of Theorem \ref{c1} and the subsequent discussion.

\begin{corolario}\label{coruno}
Let $v\in C^2 (\overline{D_{\ep}^+}\setminus\{0\})$ be solution to
$(L)$ that satisfies the finite energy condition \eqref{localfinit},
and let $g:\overline{D_{\ep}^+}\setminus\{0\}\flecha
\Sigma_K\subseteq \overline{\C}$ denote its developing map, which is
a local diffeomorphism. Then:
 \begin{enumerate}
 \item
The image $g(I_{\ep}^+)$ lies on a circle $\cC_1$, and the image
$g(I_{\ep}^-)$ lies on another circle $\cC_2$, such that $\cC_1\cap
\cC_2\neq \emptyset$ (possibly $\cC_1=\cC_2$).
 \item
The geodesic curvature of $\cC_i$, when parametrized as $g(s,0)$, is
constant of value $-c_i/2$, for the metric $ds_K^2$ in
\eqref{metrik}.
 \item
$g$ extends continuously to the origin, with $g(0)\in \cC_1\cap
\cC_2\subset \overline{\C}$.
 \item
The Schwarzian map $Q=v_{zz}-v_z^2/2$ extends holomorphically to
$D_{\ep}^*$ with $Q(\bar{z})=\overline{Q(z)}$, and has at the origin
at most a pole of order two.
 \item
If $K=0$, then $g(0)\in \C$. If $K=-1$ and $g(0)\in \partial \D
\equiv \S^1$, then $\cC_1$ and $\cC_2$ are tangent at $g(0)$, and
are \emph{not} arcs of horocycles.
 \end{enumerate}
\end{corolario}

From this result, it is not difficult to classify the conformal
metrics of constant curvature $K$ and finite area on $\R_+^2$ that
have a finite number of boundary singularities on the real axis, and
constant geodesic curvature along each boundary arc. From an
analytical point of view, this corresponds to classifying the
solutions $v\in C^2 (\overline{\R_+^2}\setminus\{q_1,\dots,
q_{n-1}\})$ with $\int_{\R_+^2} e^v<\8$ to the Neumann problem

\begin{equation}\label{+s}
\left\{
\begin{array}{lll}
\Delta v+2 K e^v=0&\qquad\qquad&\text{in }\r^2_+=\{(s,t)\in\r^2:\ t>0\},\\[3mm]
{\displaystyle \frac{\partial v}{\partial t}}=c_j\, e^{v/2}&&\text{on } I_j\subset \R\equiv \partial\r^2_+, \hspace{1cm} c_j\in \R,\\[3mm]
\end{array}
\right.\end{equation} where $I_j:=(q_j,q_{j+1})$, $j=0,\dots, n-1$
and $q_0=-\8<q_1<\dots <q_{n-1}<q_{n}=\8$.

There are obvious examples of this type of conformal metrics on
$\R_+^2$. To see this, we denote $$\tilde{\Sigma_K}=\left\{\begin{array}{ll}\Sigma_K & \text{ if}\ K=1,0,\\
&\\
\overline{\Sigma_K}\equiv \overline{\D} & \text{ if}\
K=-1,\end{array}\right.$$ and we consider a polygon in
$\tilde{\Sigma_K}\subseteq\overline{\C}$ whose edges are circular
arcs, and a conformal mapping from $\R_+^2$ into the region bounded
by it (there are two such regions if
$\tilde{\Sigma_K}=\overline{\C}$).  In the case that $K=-1$ we will
allow that these polygons have some vertices at $\S^1\equiv \partial
\tilde{\Sigma_{-1}}$, as long as the edges common to any of such
vertices are tangent at the vertex, and are not pieces of
horocycles. Then, the induced metric on $\R_+^2$ from $ds_K^2$ via
this conformal mapping gives a metric in the above conditions. (That
the area is finite when $K=-1$ and the polygon has ideal vertices is
proved in the discussion after Theorem \ref{c1}).

Also, in the case $K=1$ (and so $\tilde{\Sigma_K}=\overline{\C}$) we
may compose with a suitable branched covering of $\overline{\C}$ to
obtain other conformal metrics with the desired properties, as we
explained in Subsection \ref{dosdos}.

Still, there exist many other conformal metrics on $\R_+^2$ of
finite area, constant geodesic curvature on the boundary, and a
finite number of boundary singularities. In order to explain how to
construct them, we give first some definitions.

\begin{definicion}
By a \emph{piecewise regular closed curve} in $\tilde{\Sigma_K}$ we
mean a continuous map $\alfa:\S^1\flecha \tilde{\Sigma_K}$ such that
$\alfa$ is smooth and regular everywhere except at a finite number
of points $\theta_1,\dots, \theta_n \in \S^1$. By a \emph{piecewise
regular parametrization} of $\alfa$ we mean a composition
$\beta=\alfa\circ \phi:\S^1\flecha \tilde{\Sigma_K}$, where $\phi$
is a diffeomorphism of $\S^1$.

Let $A_j\subset \S^1$, $j\in \{1,\dots, n\}$, be the arc between
$\theta_j$ and $\theta_{j+1}$ (we define
$\theta_{n+1}:=\theta_{1}$). Then $\alfa$ will be called an
\emph{immersed circular polygon} in $\tilde{\Sigma_K}$ if each
regular open arc $\alfa|_{A_j}$ has constant geodesic curvature in
$\Sigma_K$, and in the case that $K=-1$ and $\alfa (\theta_j)\in
\S^1$ the arcs $\alfa|_{A_j}$ and $\alfa|_{A_{j-1}}$ are tangent at
$\alfa(\theta_j)$ and are not pieces of horocycles.
\end{definicion}
Observe that we allow the curve $\alfa$ to have self-intersections,
even along each regular arc $A_j\subset \S^1$. We now introduce a
concept from differential topology.

\begin{definicion}
A piecewise regular closed curve $\alfa:\S^1\flecha
\tilde{\Sigma_K}$ is \emph{Alexandrov embedded} (or simply
\emph{A-embedded}) if there exists a continuous map
$G:\overline{\D}\flecha \tilde{\Sigma_K}$ such that $G\in
C^2(\overline{\D}\setminus \{p_1,\dots,p_n\})$ for some
$p_1,\dots,p_n\in \S^1$, and:
 \begin{enumerate}
\item
For every $z\in \D$ it holds that $G(z)\in \Sigma_K$ and $G$ is a
local diffeomorphism around $z$.
 \item
$G|_{\S^1}:\S^1\flecha \tilde{\Sigma_K}$ is a piecewise regular
parametrization of $\alfa$.
 \end{enumerate}
\end{definicion}

\begin{ejemplo}
Any circular polygon without self-intersections in
$\tilde{\Sigma_K}$ is A-embedded. Also, given two points $p,q\in
\C$, if $\gamma_1$ (resp. $\gamma_2$) is an oriented geodesic arc
from $p$ to $q$ (resp. $q$ to $p$), then $\gamma_1 \cup \gamma_2$ is
a circular polygon, which is not A-embedded in $\C$, but it is
A-embedded in $\overline{\C}$. Two further examples are given in
Figure 1.
\end{ejemplo}

\begin{figure}[h]
  \begin{center}
  \begin{tabular}{cc}
   \includegraphics[clip,width=5cm]{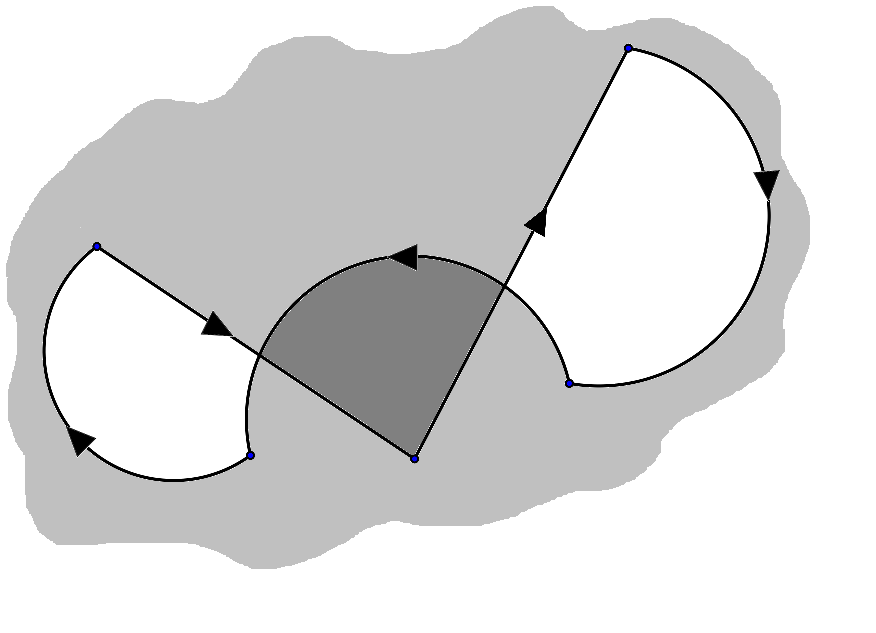} & \includegraphics[clip,width=5cm]{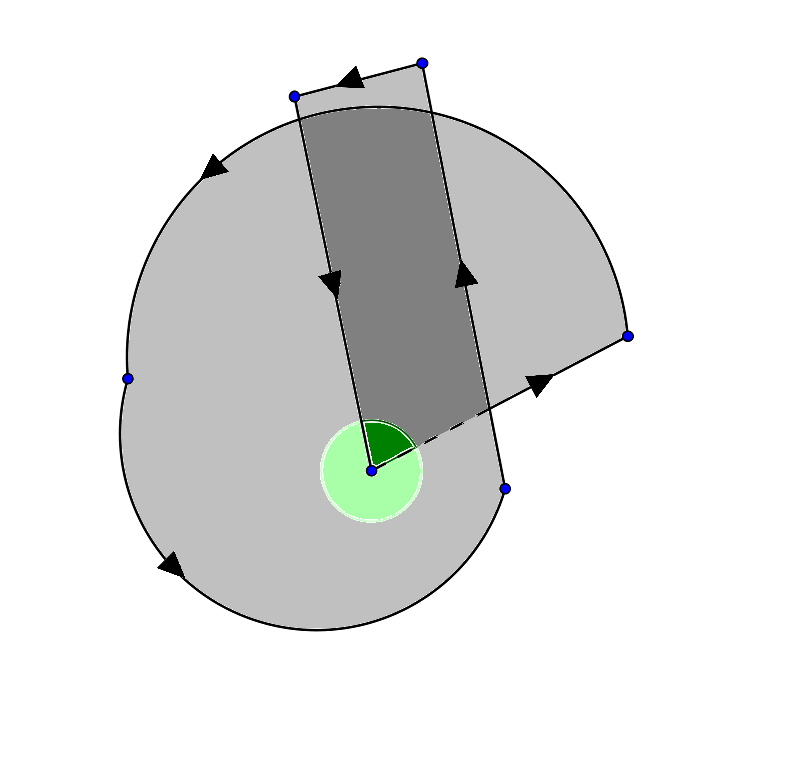}
   \end{tabular}
\caption{Two examples of circular polygonal contours that are not
embedded. The first one is A-embedded in $\overline{\c}$ but not
A-embedded in $\C$. The second one is A-embedded in $\C$ and
$\overline{\c}$. The green circle indicates that the angle at that
vertex is greater than $2\pi$.}
 \end{center}
\end{figure}

We can now associate a conformal metric of constant curvature $K$ in
$\R_+^2$ to any immersed circular polygon that is A-embedded.
Indeed, let $d\sigma^2$ denote the metric $d\sigma^2 =G^* (ds_K^2)$
induced on $\Gamma:=\overline{\D}\setminus\{p_1,\dots, p_n\}$. It is
then clear that $\Gamma$ with the complex structure induced by
$d\sigma^2$ is conformally equivalent to
$\overline{\D}\setminus\{\theta_1,\dots,\theta_n\}$ for some
$\theta_1,\dots,\theta_n\in \S^1$, or alternatively, to
$\overline{\R_+^2}\setminus\{q_1,\dots, q_{n-1}\}$ for some
$q_1<\cdots<q_{n-1}$. Consequently, $d\sigma^2$ produces on $\R_+^2$
a conformal metric $ds^2=e^v |dz|^2$ on
$\overline{\R_+^2}\setminus\{q_1,\dots, q_{n-1}\}$ that has finite
area, constant geodesic curvature on each boundary component, and
$n-1$ boundary singularities along the real axis (we have $n$
singularities if we also count the one placed at $\8$).

\begin{definicion}\label{circo}
Any such metric $ds^2$ on $\R_+^2$ will be called a \emph{circular
polygonal metric} on $\R_+^2$.
\end{definicion}

\begin{nota}
The metrics on $\R_+^2$ that we considered after equation
\eqref{+s}, starting from a polygon in $\tilde{\Sigma_K}$ whose
edges are circle arcs, are examples of circular polygonal metrics.
In that situation, the map $G$ is given by an adequate conformal
equivalence from $\D$ into the region bounded by this polygon. Note
that if $K=1$, the freedom of composing with suitable branched
coverings of $\overline{\C}$ only gives different choices of $G$
associated to the Alexandrov embedded polygonal boundary. Thus, even
the metrics involving such branched coverings are trivially circular
polygonal metrics as in Definition \ref{circo}.
\end{nota}

Once here, we have the following consequence of Corollary
\ref{coruno}:

\begin{corolario}\label{cordos}
Let $ds^2=e^v |dz|^2$ be a conformal metric of constant curvature
$K$ and finite area in $\R_+^2$. Assume that $ds^2$ extends smoothly
to $\overline{\R_+^2}\setminus \{q_1,\dots,q_{n-1}\}$ for some
$q_1<\cdots < q_{n-1} \in \R$, so that the geodesic curvature of
each boundary arc in $\R\equiv \parc \R_+^2$ is constant. Then
$ds^2$ is a circular polygonal metric in $\R_+^2$.
\end{corolario}
\begin{proof}
Let $g:\R_+^2\equiv \C_+\flecha \Sigma_K\subseteq\overline{\C}$ be
the developing map of $v$, let $\Psi:\D\flecha \C_+$ be a Möbius
transformation giving a conformal equivalence, and define $G:=g\circ
\Psi:\D\flecha \Sigma_K$. Then, by Corollary \ref{coruno} $G$
extends continuously to $\overline{\D}$, and is a local
diffeomorphism around each $z\in \D$. Again by Corollary
\ref{coruno}, it is clear that $G|_{\S^1}$ is a piecewise regular
parametrization of an immersed circular polygon $\alfa$ in
$\tilde{\Sigma_K}$, and that $G\in C^2
(\overline{\D}\setminus\{p_1,\dots, p_n\})$, where the points $p_j
\in \S^1$ are given by $\Psi (p_j)=q_j$ for $j=1,\dots, n-1$, and
$\Psi(p_n)=\8$. Hence, $\alfa$ is Alexandrov-embedded.

Finally, since $g$ is a local isometry (see Remark \ref{remiso}), we
conclude that $ds^2=e^v |dz|^2$ is indeed a circular polygonal
metric on $\R_+^2$.
\end{proof}

\begin{nota}
Corollary \ref{cordos} together with Theorem \ref{mainth1} explain
the geometric interpretation of the canonical solutions that we
pointed out without proof in Subsection \ref{dosdos}.
\end{nota}

Corollary \ref{cordos} provides a satisfactory geometric
classification of all the finite area solutions to problem
\eqref{+s}. Let us now describe such solutions from an analytic
point of view, using for that Corollary \ref{coruno} and some
classical arguments of the conformal mapping problem from the upper
half-plane to a circular polygonal domain in $\C$, see \cite{Neh}
for instance. We shall focus on the $K=1$ case, although many of the
next statements also hold when $K\leq 0$.

\begin{corolario}\label{corfor}
Let $v\in C^2(\overline{\R_+^2}\setminus\{q_1,\dots, q_{n-1}\})$ be
a solution to \eqref{+s} for $K=1$ that satisfies $\int_{\R_+^2} e^v
<\8$. Then the Schwarzian map $Q=v_{zz}-v_z^2/2$ of $v$ is given by
 \begin{equation}\label{qmuchos}
Q=\sum_{i=1}^{n-1} \left(\frac{\alfa_i}{(z-q_i)^2} +
\frac{\beta_i}{z-q_i}\right)
 \end{equation}
where $\alfa_i,\beta_i \in \R$ with $\alfa_i\leq 1/2$, $i=1,\dots,
n-1$, satisfy the following conditions:
 \begin{equation}\label{adcond} \sum_{i=1}^{n-1} \beta_i =0,
\hspace{1cm} \sum_{i=1}^{n-1} (\alfa_i + q_i \beta_i) \leq 1/2.
 \end{equation}
Conversely, if $Q$ is as in \eqref{qmuchos}-\eqref{adcond}, then
there is a solution $v\in C^2(\overline{\R_+^2}\setminus\{q_1,\dots,
q_{n-1}\})$ to problem $\eqref{+s}$ for $K=1$ that satisfies
$\int_{\R_+^2} e^v <\8$. Moreover, the family of such solutions with
the same $Q$ is generically three-dimensional.
\end{corolario}
\begin{proof}
Let $v$ be a solution to problem \eqref{+s} satisfying
$\int_{\R_+^2} e^v <\8$. By Corollary \ref{coruno}, the Schwarzian
map $Q$ of $v$ is holomorphic on $\C\setminus \{q_1,\dots,
q_{n-1}\}$, and has at most a pole of order two at each $q_i$. So,
clearly $Q$ is of the form
 \begin{equation*}
Q=\sum_{i=1}^{n-1} \left(\frac{\alfa_i}{(z-q_i)^2} +
\frac{\beta_i}{z-q_i}\right) + p(z)
 \end{equation*}
for some $\alfa_i,\beta_i\in \R$, $i=1,\dots, n-1$, and for some
polynomial $p(z)$ with real coefficients.

Let now $g:\C_+\flecha \overline{\C}$ denote the developing map of
$v$, which satisfies $\{g,z\}=Q$. As $v$ has finite area around each
$q_i$, by Theorem \ref{c1} we know that $g$ is a Möbius
transformation of a function of the form $(z-q_i)^{\landa} F(z)$ or
$F(z)+\log (z-q_i)$ near each $q_i$, where $F$ is holomorphic on a
punctured neighborhood of $q_i$, and has at worst a pole at $q_i$.
Noting that the Scwharzian derivative is invariant by Möbius
transformations, a simple computation shows that the coefficient
$\alfa_i$ in \eqref{qmuchos} satisfies $\alfa_i\leq 1/2$.

The rest of restrictions come from the fact that, by finite area,
the holomorphic quadratic differential $Qdz^2$ has at $\8$ at most a
pole of order two. If we let $w=-1/z$, then by conformal invariance
$Q(z) dz^2 = \tilde{Q}(w) dw^2$ where
$$\tilde{Q}(w)=\frac{1}{w^4} Q(-1/w).$$ So, again by Theorem \ref{c1} and the previous computation,
the finite area condition at infinity is that there exists
$\lim_{w\to 0} w^2 \tilde{Q}(w) = \alfa_n$ for some $\alfa_n\in
(-\8,1/2]$. By computing the first terms in the Taylor expansion of
$Q(-1/w)$, we easily see that this happens if and only if $p=0$ and
$\alfa_i,\beta_i$ satisfy the conditions \eqref{adcond}. This
completes the first part of the proof.

Conversely, let $Q$ be as in \eqref{qmuchos}-\eqref{adcond}, and let
$g$ be a solution to $\{g,z\}=Q$ in $\C_+$. By construction, $g$ is
a locally injective meromorphic function on $\C_+$, unique up to
Möbius transformations, and which extends smoothly to
$\overline{\C_+}\setminus\{q_1,\dots, q_{n-1},\8\}$. As $Q$ is real
on the real axis, we deduce from the equation $\{g,s\}=Q(s)$ on $\R$
that $g(s)$ lies on a circle in $\overline{\C}$ for each interval in
$\R\setminus \{q_1,\dots, q_{n-1}\}$. All of this shows that the map
$v\in C^2(\overline{\R_+^2}\setminus\{q_1,\dots,q_{n-1}\})$ given by
$$e^v=\frac{4|g'|^2}{(1+|g|^2)^2}$$ is a solution to \eqref{+s} for
$K=1$. We only have left to show that $\int_{\R_+^2} e^v <\8$, for
what we only need to prove this condition around each $q_i$ and
around $\8$.

Let us fix $q_i$, $i\in \{1,\dots, n-1\}$, and consider the complex
ODE $y''+ \frac{1}{2} Q y =0$. As $Q$ has at worst a pole of order
two at $q_i$, it is a classical result that a fundamental system of
solutions $(y_1,y_2)$ of this equation around $q_i$ is
$$y_1(z)= (z-q_i)^{\landa_1} \, a_1(z), \hspace{1cm} y_2(z)=
(z-q_i)^{\landa_2} \, a_2(z)+ k y_1(z) \log (z-q_i),$$ where $k\in
\C$, $a_1(z),a_2(z)$ are holomorphic on a neighborhood of $q_i$ with
$a_i(0)\neq 0$ for $i=1,2$, and $\landa_1,\landa_2$ are solutions of
the \emph{indicial equation} $$\landa^2-\landa + \frac{\alfa_i}{2}
=0.$$ Here $\alfa_i$ is the coefficient of $Q$ in $q_i$ given by
\eqref{qmuchos}. Note that from $\alfa_i\leq 1/2$ we deduce that
$\landa_1,\landa_2\in \R$, and we may assume that $\landa_1\leq
\landa_2$. Therefore, $k\neq 0$ if and only if $\landa_2-\landa_1
\in \N$.

Also, it is a classical result that the quotient $y_2/y_1$ provides
a solution to $\{g,z\}=Q$, that is,  a developing map for the
solution $v$. Thus, depending on whether $\landa_2-\landa_1\in \N$
or not, $g$ is of the form
$$g(z)= F(z)+ \log (z-q_i) \hspace{1cm} \text{ or }
\hspace{1cm} g(z)=(z-q_i)^{\landa_2-\landa_1} F(z) $$ for some
meromorphic function $F$ around $q_i$ such that
$F(\bar{z})=\overline{F(z)}$. By Theorem \ref{c1} and its subsequent
discussion, we see then that $\int e^v <\8$ on the half-disk
$D^+(q_i,\ep)\subset \R_+^2$ for $\ep>0$ small enough (note that we
are assuming that $K=1$).

The same argument can be done at $\8$, this time using the
additional conditions \eqref{adcond} and the conformal change
$w=-1/z$, as we did before. This concludes the proof of existence.

Finally, observe that the solution $g$ to $\{g,z\}=Q$ is unique up
to Möbius transformations, so there is a real $6$-parameter family
of possible choices for $g$. As the developing map of a solution $v$
to \eqref{lio} is defined up to the change \eqref{isom}, we obtain
generically a $3$-parameter family of solutions to \eqref{+s} for
$K=1$ with the same $Q$. This concludes the proof of the result.

\end{proof}

\def\refname{References}

\end{document}